\documentclass[hidelinks, 11pt]{article}
\usepackage{amsthm}
\usepackage{amsfonts}
\usepackage{amssymb}
\usepackage[mathscr]{euscript}
\usepackage{amsmath}
\usepackage[english, activeacute]{babel} 
\usepackage{graphicx} 
\usepackage[margin =1.2 in]{geometry}
\usepackage[dvipsnames]{xcolor} 
\usepackage{hyperref} 
\usepackage{tcolorbox} 
\usepackage{captdef}
\usepackage{caption}
\usepackage[toc,page]{appendix}
\usepackage{color}
\usepackage{pdfpages}
\usepackage{relsize}
\usepackage[utf8x]{inputenc}

\newtheorem{prop}{Proposition}
\newtheorem{lem}{Lemma}
\newtheorem{thm}{Theorem}

\newtheorem{conj}{Conjecture}

\theoremstyle{remark}

\theoremstyle{definition}
\newtheorem{defi}{Definition}

\begin{document}

\title{Projective Self-dual polygons in higher dimensions}
\author{Ana C. Chavez Caliz}

\maketitle

\begin{abstract}
Motivated by a question from V. Arnold about self-dual curves in projective spaces, we study ${\cal M}_{m,n,k}$: the moduli space of $m$-self-dual $n$-gons in $\mathbb{P}^k$. This paper lays out an explicit construction of self-dual polygons, and for specific cases of $n$ and $m$, provides the dimension of ${\cal M}_{m,n,k}$. We include a conjecture about the Pentagram map in higher dimensions that generalizes Clebsch's theorem, which states that every pentagon in $\mathbb{RP}^2$ is invariant under the Pentagram map.\\
\end{abstract}

\tableofcontents

\section{Introduction}
	In his book ``\textit{Arnold's problems}'' \cite{arnold}, Vladimir Arnold shares a collection of questions formulated at the beginning of his seminars organized during his time in Moscow and Paris for over $40$ years. One of these problems, stated in $1994$, goes as follows:

	\begin{center}
		\textbf{1994-17.} Find all projective curves projectively equivalent to their duals. \textit{The answer seems to be unknown even in $\mathbb{RP}^2$}.
	\end{center}

	Given a vector space $V$, the \emph{\textbf{projective space}}\footnote[2]{Throughout the text, we will also denote projective spaces as $\mathbb{P}^k$ to emphazise the dimension of the space: $k$.} $\mathbb{P}$ consists of all $1$-dimensional subspaces of $V$. The projective dual space $\mathbb{P}^*$ is then the projectivization of $V^*$, where $V^*$ is the vector space given by all linear functions over $V$. To each subspace in $V$, we can associate it to its annihilator in $V^*$. In particular, this association provides a one-to-one correspondence between hyperplanes in $\mathbb{P}$ and points in $\mathbb{P}^*$.\\

	In the case of polygons, a \emph{\textbf{closed $\boldsymbol{n}$-gon}} $P$ in a projective space of dimension $k$ is a sequence of vertices $A_1, A_3, A_5, \ldots \in \mathbb{P}^k$ such that $A_i = A_{i+2n}$ for all $i$. Its \emph{\textbf{dual polygon}} $P^*$ is defined as the sequence of vertices $B_k^*, B_{k+2}^* , B_{k+4}^*, \ldots \in (\mathbb{P}^k)^*$, where 
$$B_{i} := \text{span} \{A_{i-(k-1)}, A_{i-(k-3)}, \ldots ,A_{i+(k-3)}, A_{i+(k-1)} \}.$$

	Given $m \in \mathbb{Z}$, an $n$-gon $P=(A_1, \ldots , A_{2n-1})$ and its dual polygon $P^* = (B_k^*, \ldots , B_{k +2(n-1)}^*)$, $P$ is called \emph{\textbf{$\boldsymbol{m}$-self-dual}} if there is a projective transformation $\hat{f} : \mathbb{P}^k \rightarrow (\mathbb{P}^k)^*$ such that $\hat{f}(A_i) = B_{i+m}^*$ for all $i$. The \emph{\textbf{moduli space of $\boldsymbol{m}$-self-dual $\boldsymbol{n}$-gons in $\boldsymbol{\mathbb{P}^k}$}} denoted as $\boldsymbol{{\cal M}_{m,n,k}}$ is the set of all $m$-self-dual $n$-gons in $\mathbb{P}^k$, under the action of $\text{PGL}(V)$: two $n$-gons $P = (A_1, A_3, \ldots, A_{2n-1})$, $P'= (A'_1, A'_3, \ldots , A'_{2n-1})$ belong to the same equivalent class if there is a projective transformation $\psi : \mathbb{P}^k \rightarrow \mathbb{P}^k$ such that $\psi(A_i)=A'_i$ for all $i$.\\

Thus, the discrete version of Arnold's question for closed curves becomes:

	\begin{center}
		\textit{Find all $m$-self-dual projective polygons.}
	\end{center}

	In \cite{selfdual}, D. Fuchs and S. Tabachnikov study the space of $m$-self-dual polygons in a 2-dimensional projective space (either $\mathbb{RP}^2$ or $\mathbb{CP}^2$). Theorem 1 of \cite{selfdual} states that:

\begin{thm}[D. Fuchs, S. Tabachnikov]\label{Fuchs-Tabach}
If $(m,n)=1$ then ${\cal M}_{m,n,2}$ consists of one point, the class of a regular $n$-gon. If $m\leq n$, $(m,n)> 1$ and $n \neq 2m$ then $\dim {\cal M}_{m,n,2}=(m,n)-1$. Finally, $\dim {\cal M}_{m,2m,2} = m-3$ and $\dim {\cal M}_{n,n,2}=n-3$.
\end{thm}

The work done in \cite{selfdual} includes, among other results, the following two:
	\begin{enumerate}
	\item Any pentagon in $\mathbb{P}^2$ is $5$-self-dual (which is the smallest\footnote[3]{Smallest number of vertices.} interesting case as all triangles are projectively equivalent in $\mathbb{P}^2$, and similarly for quadrilaterals),
	\item If $n$ is odd, and $P$ is an $n$-gon inscribed in a conic and circumscribed about a conic (that is, $P$ is a Poncelet polygon), then $P$ is $n$-self-dual.
	\end{enumerate}

	In addition, \cite{selfdual} extends some of the results for closed polygons to closed curves, and provides examples of self-dual curves (like Radon curves) together with explicit formulas.\\

The next natural step is to generalize the results in \cite{selfdual} to higher-dimensional projective spaces. For instance, one can ask:

	\begin{enumerate}
	\item What is the dimension of ${\cal M}_{m,n,k}$?
	\item Are $(k+3)$-gons $(k+3)$-self-dual in $\mathbb{P}^k$?
	\item If $P$ is an $n$-gon inscribed in a quadric and circumscribed to a quadric, is $P$ $n$-self-dual?
	\item Are periodic billiard trajectories\footnote[9]{When the dimension is 2, these correspond precisely to Poncelet polygons.} in ellipsoids self-dual? 
	\item What can we say about the continuous case, in particular for closed self-dual curves?
	\end{enumerate}
	
	Most of this paper will focus on studying $m$-self-dual $n$-gons in $\mathbb{CP}^k$. The space of $n$-gons in $\mathbb{RP}^k$ have been extensively studied (\cite{projective-gale}, \cite{gale-original}, \cite{geometry-grassmannians}, \cite{chow-quotients} to mention just a few). Polygons in projective spaces are connected to Linear Difference equations and Frieze Patterns (\cite{gale}). \\

	The scenario in higher dimensions is more intricate than the two-dimensional case, as the classification of non-degenerate bilinear forms in arbitrary dimensions increases in complexity. However, we were able to obtain partial answers to some of the questions posted above. 

\begin{thm} \label{(m,n)=1}
If $(m,n)=1$, then there are finitely-many $m$-self-dual $n$-gons in $\mathbb{CP}^k$, and so $\dim {\cal M}_{m,n,k}= 0$. These polygons correspond to the projective equivalence class of regular $n$-gons \footnote[4]{Following H. S. M. Coxeter's definition (\cite{coxeter}) a polygon $P=(A_1, \ldots ,A_{2n-1})$ in an Euclidean space $\mathbb{E}^k$ is \emph{\textbf{regular}} if there is an isometry $S$ of $\mathbb{E}^k$ such that $A_{2\ell+1} = S^{\ell}(A_1)$. Since every projective space $\mathbb{RP}^k$ contains an affine plane $\mathbb{E}^k$, it makes sense to talk about the class of regular $n$-gons in $\mathbb{RP}^k$. Figure \ref{regular} shows two examples of regular polygons.} in $\mathbb{C}^k$. 
\end{thm}

	\begin{center}
%	\centering
	\includegraphics[scale=0.38]{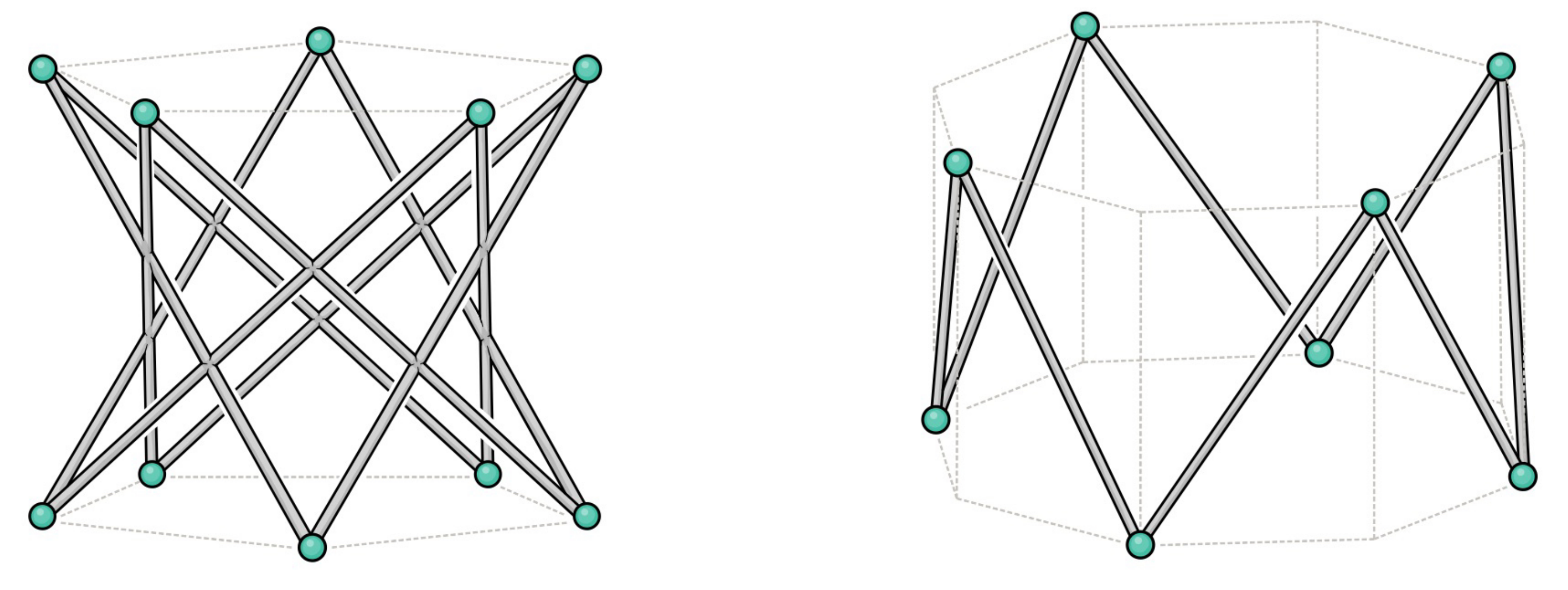}
	\figcaption{Examples of regular polygons in $\mathbb{R}^3$, which are $m$-self-dual, for any $m$. A 3-dimensional model for each polygon, from left to right, can be found using the following links: \url{https://www.geogebra.org/3d/pejx6wdj} and \url{https://www.geogebra.org/3d/tx3rqkrg}.}
	\label{regular}
	\end{center}

\begin{thm} \label{m=n} If $n\equiv k+1 \pmod 2$ then 
$\dim {\cal M}_{n,n,k} = \dfrac{k(n-k-1)}{2}$.\\
\end{thm}

	\begin{center}
%	\centering
	\begin{minipage}{7cm}
		\includegraphics[scale=0.33]{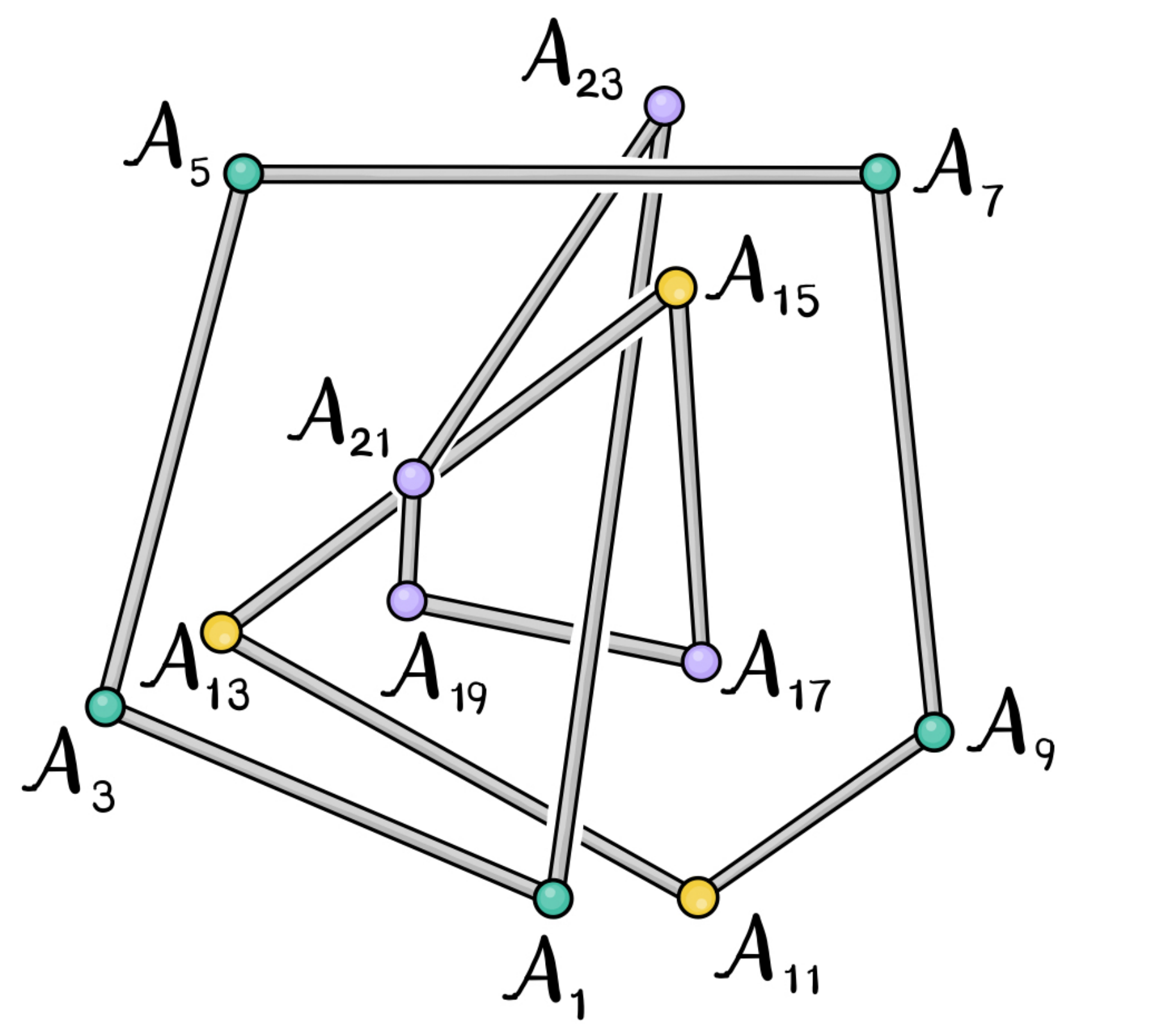}	
	\end{minipage}
	\begin{minipage}{7cm}
		\figcaption{Example of a $12$-self-dual $12$-gon in $\mathbb{RP}^3$. The dimension of ${\cal M}_{12,12,3}$ is given by Theorem \ref{m=n}. A $3$-dimensional model can be found here: \url{https://www.geogebra.org/m/pfzr3trx}.}
	\end{minipage}
	\label{m=n-example}
	\end{center}

\begin{thm} \label{r=2}
The moduli space of $m$-self-dual $2m$-gons in $\mathbb{CP}^k$ consists of $\left\lfloor \frac{k}{2} \right\rfloor$ components $${\cal M}_{m,2m,k}= \bigsqcup_{f=1}^{\lfloor k/2 \rfloor} {\cal M}_f.$$ 

\begin{itemize}
	\item If $k<m$, then for every $1\leq f \leq \lfloor k/2 \rfloor$ $$\dim {\cal M}_f = \dfrac{mk}{2} - 4f^2 + 2kf - {k+1 \choose 2}.$$   
	The dimension of the biggest of all these components, ${\cal M}_{f_0}$, depends on the congruence class of $k \pmod 4$ as follows:
		\begin{itemize}
		\item[$\circ$] If $k\equiv 0 \pmod 4$, then $f_0=k/4$, and $\dim {\cal M}_{f_0}= \dfrac{k(n-k-2)}{4},$
		\item[$\circ$] If $k\equiv \pm 1 \pmod 4$, then $f_0=\frac{k\pm 1}{4}$, and $\dim {\cal M}_{f_0}= \dfrac{nk-(k+1)^2}{4},$\\
		\item[$\circ$] If $k\equiv 2 \pmod 4$, then $f_0=\frac{k\pm 2}{4}$, and $\dim {\cal M}_{f_0}= \dfrac{nk-(k^2+2k+4)}{4}.$\\
		\end{itemize}

	\item If $k>m$, then for every $1\leq f \leq \lfloor k/2 \rfloor$ $$\dim {\cal M}_f = \dfrac{m(k+2)}{2} - 4f^2 + 2kf - {k+1 \choose 2}.$$   
	The dimension of the biggest of all these components, ${\cal M}_{f_0}$, depends on the congruence class of $k \pmod 4$ as follows:
		\begin{itemize}
		\item[$\circ$] If $k\equiv 0 \pmod 4$, then $f_0=k/4$, and $\dim {\cal M}_{f_0}= \dfrac{(k+2)(n-k)}{4},$
		\item[$\circ$] If $k\equiv \pm 1 \pmod 4$, then $f_0=\frac{k\pm 1}{4}$, and $\dim {\cal M}_{f_0}= \dfrac{n(k+2)-(k+1)^2}{4},$\\
		\item[$\circ$] If $k\equiv 2 \pmod 4$, then $f_0=\frac{k\pm 2}{4}$, and $\dim {\cal M}_{f_0}= \dfrac{n(k+2)-(k^2+2k+4)}{4}.$
		\end{itemize}
\end{itemize}
\end{thm}

	\begin{center}
%	\centering
	\includegraphics[scale=0.33]{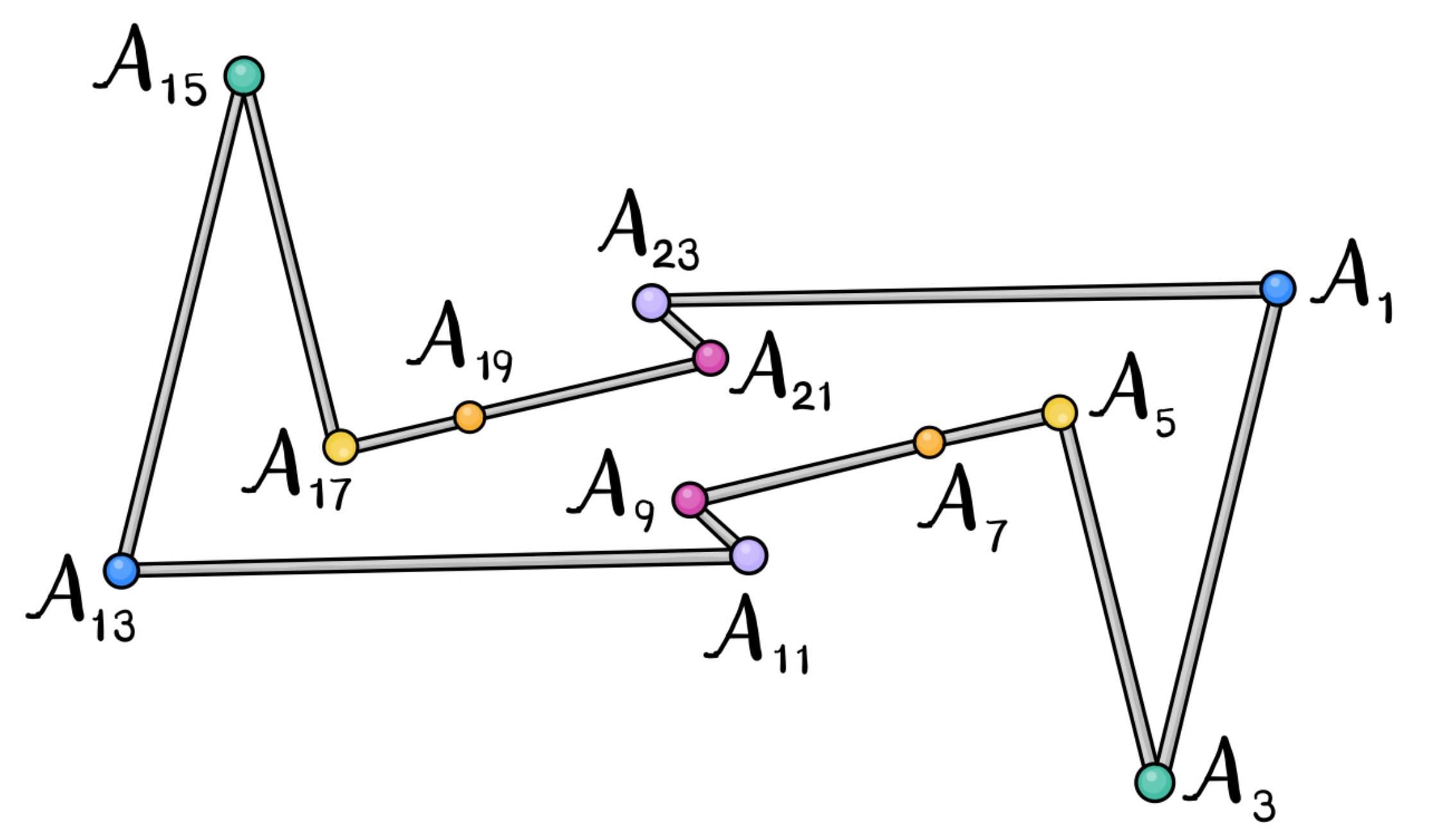}
	\figcaption{Example of a $6$-self-dual $12$-gon in $\mathbb{RP}^3$. The dimension of ${\cal M}_{6,12,3}$ is given by Theorem \ref{r=2}. A $3$ dimensional model of this polygon can be found here: \url{https://www.geogebra.org/m/szykrkvg}.}
	\label{r=2-example}
	\end{center}

\begin{thm} \label{r=3}
Let $m,n$ such that $n= 3(m,n)$. Then the moduli space of $m$-self-dual $n$-gons consists of $${ \lfloor k/2 \rfloor +2 \choose 2 }-1$$ components. These components can be sorted into $\lfloor k/2 \rfloor$ groups: ${\cal M}_1, \ldots, {\cal M}_{\lfloor k/2 \rfloor}$. \\
Each group ${\cal M}_s$ consists of $s+1$ components of dimension 
$$\dfrac{(m,n)(k+D(m,n,k))}{2} - 3s^2 + (2k+1)s - {k+1 \choose 2},$$
where $ 0 \leq D(m,n,k) \leq 3$\footnote[1]{The precise value of $D(m,n,k)$ is stated in Lemma \ref{compute-drama}, subsection \ref{novela}.}.\\

The class with the biggest dimension is ${\cal M}_{\lceil k/3\rceil}$, and the dimension depends on the congruence class of $k \pmod 3$:

$$
\dim {\cal M}_{\lceil k/3 \rceil} = 
\begin{cases}
\vspace*{0.5cm}
\dfrac{n(k+D(m,n,k)) - k(k+1) }{6} \hspace{1.1 cm} \text{ if } k \equiv 0, 2 \pmod 3 , \\

\dfrac{n(k+D(m,n,k)) - (k^2+k+4)}{6} \hspace{0.45cm} \text{ if } k \equiv 1 \pmod 3 .
\end{cases}
$$

\end{thm}

	\begin{center}
%	\centering
	\includegraphics[scale=0.33]{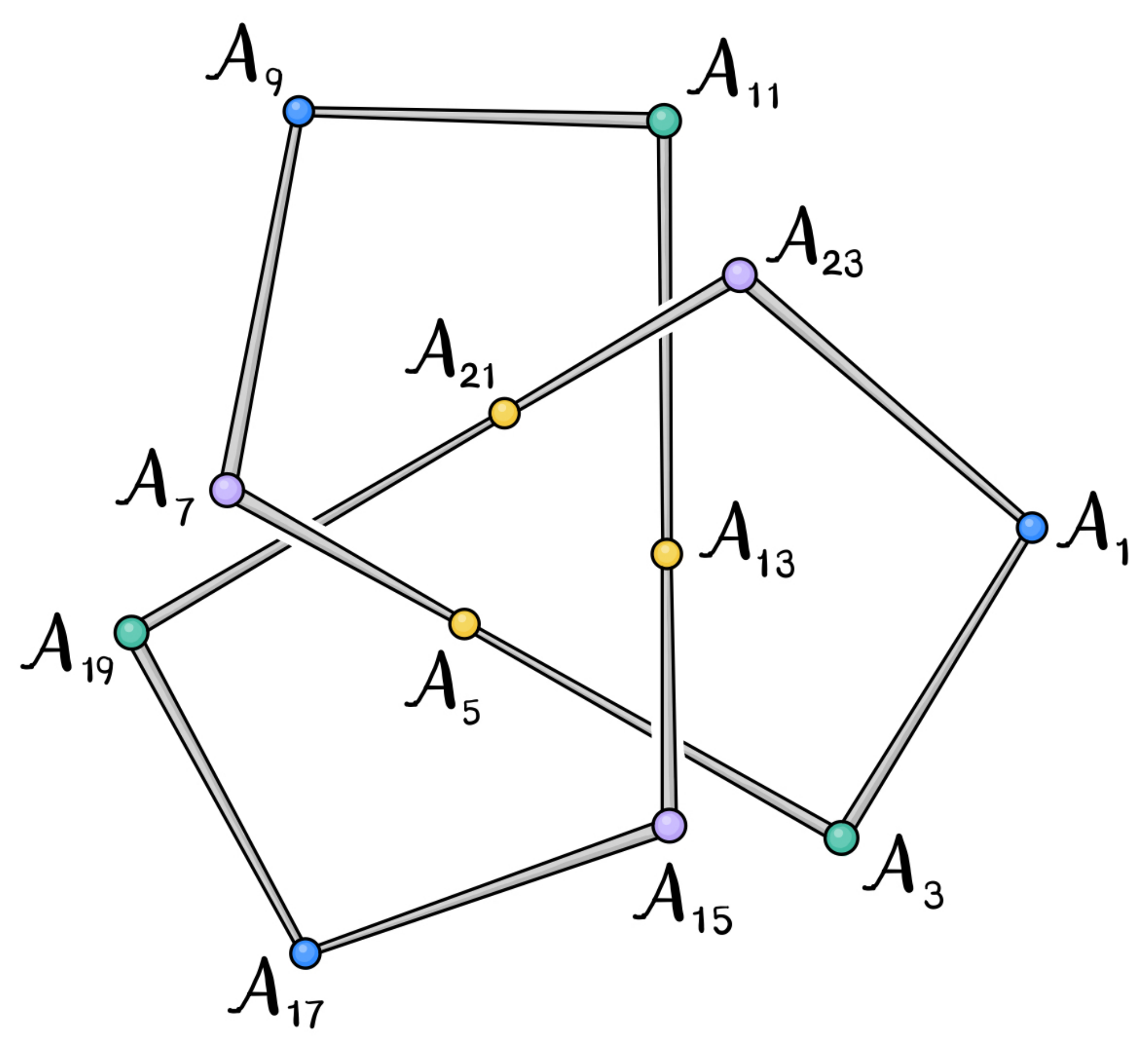}
	\figcaption{A $4$-self-dual $12$-gon in $\mathbb{RP}^3$ (described by Theorem \ref{r=3}). The vertices $A_{i}, A_{i+8}, A_{i+16}$ form an equilateral triangle contained in a horizontal plane $z=c$. The $3$-dimensional model can be found here: \url{https://www.geogebra.org/3d/vvwfsa7k}.}
	\label{r=3-example}
	\end{center}

\begin{thm}\label{k+3}
All $(k+3)$-gons in $\mathbb{CP}^k$ are $(k+3)$-self-dual.
\end{thm}

\begin{thm}\label{m=0} If $k$ is odd, then
$$\dim {\cal M}_{0, n, k} = \dfrac{(k+1)(n-k-2)}{2}.$$
\end{thm}

\clearpage
For instance, one can use Theorems \ref{(m,n)=1}, \ref{r=2}, \ref{r=3}, \ref{m=0} to provide the dimension of the space of self-dual $12$-gons in $\mathbb{RP}^3$ for different values of $m$:

\begin{center}
\begin{tabular}{c c c}
$\boldsymbol{m}$ \hspace{0.2 cm} & \textbf{Theorem} & $\boldsymbol{\dim {\cal M}_{m, 12, 3}}$\\
\hline
$4$ \hspace{0.2 cm}& Theorem \ref{r=3} & $4$ \\
$6$ \hspace{0.2 cm}& Theorem \ref{r=2} & $5$ \\
$8$ \hspace{0.2 cm}& Theorem \ref{r=3} & $6$ \\
$12$ \hspace{0.2 cm}& Theorem \ref{m=n} & $12$ \\
$0$ \hspace{0.2 cm}& Theorem \ref{m=0} & $14$
\end{tabular}
\end{center}

More than providing the dimension of ${\cal M}_{m,n,k}$ in different cases, this paper provides an explicit construction of $m$-self-dual $n$-gons.\\

Here is an outline of the arguments followed to obtain the previous results: Section \ref{basics} provides the framework used to study the space ${\cal M}_{m,n,k}$: every polygon $P \in {\cal M}_{m,n,k}$ is associated with a non-degenerate bilinear form $F$ defined up to a non-zero factor. More importantly, this bilinear form $F$ gives rise to a projective transformation $G$ that acts on the vertices of $P$ as a cyclic permutation. In particular, the projective transformation $G$ can be used to reconstruct the self-dual polygon $P$. Section \ref{basics} splits in two parts: when the associated bilinear form $F$ is symmetric and when it is not. In both cases, the method to obtain the dimension of $m$-self-dual $n$-gons is the same: pick the ``first'' vertices of the polygon $P$, and use the projective transformation $G$ to construct the rest. Therefore: 

$$\dim {\cal M}_{m,n,k} = \substack{\text{Degrees of freedom} \\ \text{we have to choose} \\\text{the first vertices}} - \substack{\text{Dimension of group of} \\ \text{linear transformations}\\ \text{that preserve } F.}$$

In Section \ref{gale-transform}, we use the Gale transform for polygons to extend the results known to more cases.\\ 

The paper concludes with Section \ref{Pentagram-section}, which contains a conjecture about the generalization of the Pentagram map in higher dimensions, found using computer-based mathematical experiments. The computational data supports the conjecture for up to dimension $k=200$.\\

%-----------------------------------------------------------------------------
%-----------------------------------------------------------------------------

\section{Bilinear forms and self-dual polygons}\label{basics}

\begin{defi}
Given a $n$-gon $P = (A_1, A_3, \ldots , A_{2n-1}), A_i \in \mathbb{CP}^k$, and $\ell \in \mathbb{N}$, define the $(\ell n)$-gon, $\ell P:= (A_1, A_3, \ldots ,A_{2\ell n -1})$. A polygon is called \textbf{\emph{simple}} if $P \neq \ell P'$ for any $\ell >1$ and any polygon $P'$.
\end{defi}

Throughout this paper, we will only consider simple polygons such that any $(k+1)$-consecutive vertices are linearly independent to allow the following definition.

\begin{defi} 
Given a closed polygon $P = (A_1, A_3, ... , A_{2n-1}), A_i \in \mathbb{CP}^k$, its \textbf{\emph{dual polygon}} $P^* := \left(B_k^*, B_{k+2}^*, \ldots , B_{k+2(n-1)}^*\right), B_i^* \in (\mathbb{CP}^k)^*$ where 
$$B_{i} := \text{span} \{A_{i-(k-1)}, A_{i-(k-3)}, \ldots ,A_{i+(k-3)}, A_{i+(k-1)} \}.$$

Given $m \in \mathbb{Z}$ such that $m \equiv k-1 \pmod 2$, we say $P$ is \textbf{\emph{$\boldsymbol{m}$-self-dual}} if there is a projective transformation $\hat{f}: \mathbb{CP}^k \rightarrow (\mathbb{CP}^k)^*$ such that $\hat{f}(A_i)=B_{i+m}^*$ for every subindex $i$.
\end{defi}

There are some natural assumptions about the variables $m$, $n$, and $k$:
	\begin{itemize}
		\item $n\geq k+3$: since $\text{PGL}(k+1, \mathbb{C})$ is $(k+2)$-transitive in $\mathbb{CP}^k$, the problem of classifying $m$-self-dual $n$-gons is trivial when $n\leq k+2$.
		\item $m \equiv k-1 \pmod 2$: this condition is necessary because the indices for the vertices of $P$ are odd, and the indices of the vertices of $P^*$ have the same parity as $k$.
	\end{itemize}

%--------------------------------------------------------------------------------------------------------------------------------------------

Let $P$ be an $m$-self-dual $n$-gon, with $\hat{f}$ such that $\hat{f}(A_i)= B_{i+m}^*$. Consider the linear transformation $f: \mathbb{C}^{k+1}\rightarrow \left(\mathbb{C}^{k+1}\right)^*$ such that $f$ sends each line $A_i$ to the line $B_{i+m}^*$, and a bilinear form $F$ defined as $$F(u,v)=\langle f(u), v\rangle.$$ Both $f$ and $F$ are well-defined up to a non-zero factor.

\begin{defi}\label{g} 
Given $B \subset \mathbb{C}^{k+1}$, define the \textbf{\emph{perpendicular space}} of $B$, denoted as $\boldsymbol{B^\perp}$, to be:
$$B^{\perp}:= \{ y \in \mathbb{C}^{k+1} \hspace{0.2cm} \vert \hspace{0.2cm} F(x,y)=0, \text{ for any } x\in B\}.$$

Let $G: \mathbb{CP}^k\rightarrow \mathbb{CP}^k$ be the projective transformation $G(A):=(A^\perp)^\perp$. In matrix representation $G= F^{-1}F^{\mathsmaller T}$.
\end{defi}

\begin{lem}\label{rotation}
If $P$ is $m$-self-dual $n$-gon, then $G(A_i)=A_{i+2m}$.
\end{lem}

\begin{proof}

First, notice 
$$F(A_i, A_j)= \langle f(A_i), A_j \rangle = \langle B_{i+m}^*, A_j\rangle =0 \Leftrightarrow A_j \in B_{i+m}.$$
In particular, 
\begin{equation}\label{perpendicular}
F(A_i, A_{i+m\pm \ell})=0,
\end{equation}
for all $\ell = k-1, k-3, \ldots \geq 0$, and any subindex $i$.\\

Since $$A_i^\perp = B_{i+m}= \text{span}\{A_{i+m-(k-1)}, A_{i+m-(k-3)}, \ldots , A_{i+m+(k-3)}, A_{i+m+(k-1)} \}, $$
Then
$$G(A_i)= (A_{i}^{\perp})^{\perp} = A_{i+m-k+1}^{\perp} \cap A_{i+m-k+3}^{\perp} \cap \ldots \cap A_{i+m+k-1}^\perp = A_{i+2m}. $$
\end{proof}

Hence, $G$ makes a cyclic permutation on the vertices of $P$, and $G^r= I$, where $r=\frac{n}{(m,n)}$.\\

\begin{lem} \label{classify}
Suppose $F$ is a non-degenerate bilinear form satisfying $(F^{-1}F^{\mathsmaller T})^r=I$ for some $r \in \mathbb{N}$. Then there is an appropriate basis of $\mathbb{C}^{k+1}$ such that the matrix representation of $F$ is a direct sum of matrices: $$F= \bigoplus_{j=1}^{\mathcal{S}} H_2(e^{i\theta_j})\bigoplus I_{k+1-2{\mathcal{S}}}, $$ where

$H_2(e^{i\theta_j}) = \begin{pmatrix}
							0 & 1 \\
						e^{i\theta_j} & 0 
							\end{pmatrix}$ with $r\theta_j \in 2\pi\mathbb{Z}$ and $\theta_j \notin 2\pi\mathbb{Z}$.
\end{lem}

\begin{proof}
According to \cite{canonical}, Theorem 1.1 part (a), every complex square matrix $F$ is congruent to a direct sum of matrices of the following three types:

	\begin{enumerate}

	\item[\textbf{Type 0:}] The $\ell \times \ell$ Jordan block with eigenvalue $0$:
	$$J_\ell(0)= \begin{pmatrix}
			0 & 1 & \cdots & 0 \\
			\vdots & \ddots & \ddots & \vdots \\
			0 & \cdots & 0 & 1 \\
			0 & \cdots & 0 & 0
			\end{pmatrix},
		\hspace{0.5cm} 
		(J_1(0)=(0)).$$

	\item[\textbf{Type 1:}] The $\ell \times \ell$ matrices:
	$$\Gamma_\ell = \begin{pmatrix}
				0 & \cdots & 0 & 0 & (-1)^{\ell+1} \\
				0 & \cdots & 0 & (-1)^\ell & (-1)^\ell \\
				0 & & \text{ \rotatebox{75}{$\ddots$} } & \text{ \rotatebox{75}{$\ddots$} } & 0 \\
				0 & -1 & -1 & \cdots & 0\\
				1 & 1 & 0 & \cdots & 0\\
				\end{pmatrix} ,
			\hspace{0.5cm}
			(\Gamma_{1}=(1)).	
				$$

	\item[\textbf{Type 2:}] And the $2\ell \times 2\ell$ matrix:
	$$H_{2\ell}(\mu) = \begin{pmatrix}
	0 & I_\ell \\
	J_{\ell}(\mu) & 0
	\end{pmatrix}, 
	\hspace{0.5cm}
	H_{2}(\mu) = \begin{pmatrix}
	0 & 1 \\
	\mu & 0
	\end{pmatrix}, $$
	
	where $J_\ell$ is the $\ell \times \ell$ Jordan block with eigenvalue $\mu$, $0\neq \mu \neq (-1)^{\ell+1}$.
	
	\end{enumerate}
	
Since $F$ is non-degenerate, the direct sum cannot contain matrices of type 0.\\

Note that, if $F$ is a direct sum $F = \bigoplus_i M_i$ of square matrices $M_i$, then $$(F^{-1}F^{\mathsmaller T})^r= \bigoplus_i (M_i^{-1}M_i^{\mathsmaller T})^r.$$

Given that $(F^{-1}F^{\mathsmaller T})^r= I$, then each one of the summands of $F$ must satisfy the same property. Straightforward calculations show that $\Gamma_\ell^{-1}\Gamma_{\ell}^{\mathsmaller T} \neq I_{\ell} $ and $H_{2\ell}(\mu)^{-1}H_{2\ell}(\mu)^{\mathsmaller T} \neq I_{2\ell}$ if $\ell>1$. \\

If $\ell =1$, then 

	\begin{itemize}
	\item $(\Gamma_1^{-1} \Gamma_1^{\mathsmaller T})^r =1$ for any $r$, and
	\item $(H_2(\mu)^{-1}H_2(\mu)^{\mathsmaller T})^r = \begin{pmatrix}
	\mu^{-r} & 0 \\
	0 & \mu^{r}
	\end{pmatrix} = \begin{pmatrix}
	1 & 0 \\
	0 & 1
	\end{pmatrix} $ if and only if $ \mu = e^{\theta i}$, with $r\theta \in 2\pi \mathbb{Z}$. Moreover $0\neq \mu \neq 1$, thus $\theta \notin 2\pi\mathbb{Z}$. 
	\end{itemize}
\end{proof}

From here, we split the analysis into two cases, depending on whether $F$ is symmetric or not.

%--------------------------------------------------------------------------------------------------------------------------------------------

\subsection{$F$ symmetric}
		
\begin{prop}\label{n=m-form}
Let $P \in {\cal{M}}_{m, n, k}$. The corresponding bilinear form $F$ is symmetric if and only if $m=n$.
\end{prop}
 
\begin{proof}
The first part of the proof follows a similar argument of the proof of Proposition 2, from \cite{selfdual}.

If $F$ is symmetric, then $0=F(A_{i+m\pm \ell}, A_i)= \langle B_{i+2m \pm \ell}^*, A_i\rangle$. \\
Since any $(k+1)$-consecutive vertices of $P$ are linearly independent then, 
	\[
	\bigcap_{\substack{0\leq \ell \leq k-1 \\ \ell \equiv k-1 (\text{mod } 2)}} B_{i+2m \pm \ell} = \{ A_{i+2m} \},
	\]

meaning $A_i=A_{i+2m}$ for every subindex $i$. But $P$ is a simple polygon, therefore $m=n$.\\

On the other hand, if $n=m$, according to Lemma \ref{classify}, $$F=\bigoplus_{j=1}^{\mathcal{S}} H_2(e^{i\theta_j})\bigoplus I_{k+1-2{\mathcal{S}}},$$ with $F^{-1}F^{\mathsmaller T}=I$. \\

Notice that 

$$H_2(\mu)^{-1}H_2(\mu)^{\mathsmaller T} = \begin{pmatrix}
	\mu^{-1} & 0 \\
	0 & \mu
	\end{pmatrix} = \begin{pmatrix}
	1 & 0 \\
	0 & 1
	\end{pmatrix}, $$ 

if and only if $ \mu = 1$. But the classification from Lemma \ref{classify} indicates that $\mu\neq 1$, so there is an appropriate basis of $\mathbb{C}^{k+1}$ such that $F=I_{k+1}$.
	\end{proof}

Therefore if $m=n$, then the projective duality becomes the polar duality:
$$A=(a_0, \ldots,  a_k) \mapsto A^\perp = \{ (x_0, \ldots , x_k): a_0x_0 + \ldots + a_kx_k =0\}.$$
In particular, since $ F(A_i , A_{i+n \pm \ell}) = 0$  for all $\ell = k-1, k-3, \ldots \geq 0$ and any subindex $i$, it follows that $ A_{i+n \pm \ell}  \in A_i^{\perp} $. \\

We are now ready to prove our first result regarding the dimension of ${\cal M}_{m,n,k}$ when $m=n$, Theorem \ref{m=n}, stated in the introduction of the paper.

\begin{proof} (Of Theorem \ref{m=n})
Notice that since $m=n$, then $n\equiv k+1 \pmod 2$. Thus, the expression $k(n-k-1)/2$ is an integer number. \\

We distinguish 3 types of vertices:

	\begin{itemize}
		\item The vertices ``before'' $A_1^\perp$: $A_1, \ldots A_{n-k}$ (there are in total $\frac{n-k+1}{2}$ such vertices),
		\item The vertices in $A_1^\perp$:  $A_{n-k+2}, \ldots, A_{n+k}$ ($k$ vertices total), and
		\item The vertices ``after'' $A_1^\perp$: $A_{n+k+2}, \ldots A_{2n-1}$ (a total of $\frac{n-k-1}{2}$).
	\end{itemize}

Notice
$$\dim {\cal M}_{m,n,k} = \substack{\text{Degrees of freedom} \\ \text{we have to choose} \\A_1, A_3, \ldots A_{2n-1}} - \substack{\text{Dimension of group of} \\ \text{linear transformations}\\ \text{that preserve } F.  }$$

To calculate the degree of freedom in the choice of vertices $A_1, \ldots, A_{n-k}$:
	\begin{itemize}
		\item First, choose the vertices $A_1, \ldots , A_{n-k}$. Each one of these vertices can be chosen with degree of freedom $k$, because the sets $A_1^\perp, A_3^\perp, \ldots A_{n-k}^\perp$ don't contain any of these vertices. This means that we can choose this family of vertices from a set of dimension $\frac{k(n-k+1)}{2}$.
		\item Notice each vertex $A_{n-k + 2\ell}$ (with $\ell=1, \ldots, k$) must satisfy:
			$$A_{n-k+2\ell} \in \bigcap_{j=1}^{\ell} A_{2j -1}^\perp.$$
			
			Thus each $A_{n-k+2\ell}$ can be chosen from a $(k-\ell)$-dimensional set, and therefore this family can be chosen with degree of freedom: $$(k-1)+ (k-2) + \ldots +2 + 1 = \dfrac{k(k-1)}{2}.$$
			
		\item Finally, notice the vertices $A_{n+k+ 2\ell}$ (with $\ell=1, \ldots, \frac{n-k-1}{2}$) are completely determined once we chose the previous vertices, since
			$$A_{n+k+2\ell} \in \bigcap_{j=1}^{k} A^{\perp}_{2\ell + 2j-1}.$$
	\end{itemize}

Hence, the vertices $A_1, \ldots, A_{2n-1}$ can be chosen with degree of freedom:

$$\frac{k(n-k+1)}{2}+ \dfrac{k(k-1)}{2} = \dfrac{kn}{2}.$$

On the other hand, the orthogonal group $O(k+1, \mathbb{C})$ is precisely the one that preserves $F$. The dimension of $O(k+1, \mathbb{C})$ is $\dfrac{k(k+1)}{2}$.

Then

$$\dim {\cal M}_{m,n,k} = \dfrac{kn}{2} - \dfrac{k(k+1)}{2} = \dfrac{k(n-k-1)}{2}.$$

\end{proof}

%--------------------------------------------------------------------------------------------------------------------------------------------

\subsection{$F$ not symmetric}

If $P \in {\cal M}_{m,n,k}$ with $m\neq n$, then by Lemma \ref{classify} and Proposition \ref{n=m-form}, the corresponding bilinear form $F$ looks like

\begin{equation} \label{F}
F=\bigoplus_{j=1}^{\mathcal{S}} H_2(e^{i\theta_j})\bigoplus I_{k+1-2{\mathcal{S}}}.
\end{equation}

Together with this bilinear form we have a projective transformation $G$ (Definition \ref{g}). By Lemma \ref{rotation}, we know that $G$ acts on the vertices of $P$ by a cyclic permutation, with $G(A_i)=A_{i+2m}$. And if $r= \frac{n}{(m,n)}$, then $G^r=I$.\\

This means that the vertices of $P$ form $(m,n)$-different $r$-gons $P_i$ with vertices $$P_i= (A_i, A_{i+2m}, A_{i+4m}, \ldots , A_{i+2(r-1)m}).$$ By the extended Euclidean algorithm, $A_{1+2(m,n)}$ is also part of this list. Moreover, one can rearrange the list such that the subindices are in increasing order that is

 $$P_i= (A_i, A_{i+2(m,n)}, A_{i+4(m,n)}, \ldots , A_{i+ 2(r-1)(m,n)}),$$
with each $P_i$ completely determined by the choice of $A_i$. This observation implies Theorem \ref{(m,n)=1} stated in the introduction.\\

But if $(m,n)\neq 1$, then

$$\dim {\cal M}_{m,n,k} = \substack{\text{Degrees of freedom} \\ \text{we have to choose} \\A_1, A_3, \ldots A_{2(m,n)-1}} - \substack{\text{Dimension of group of} \\ \text{linear transformations}\\ \text{that preserve } F.  }$$

The following two subsections will explore how to count each summand of the previous equation. First, we count the degrees of freedom to choose the first $(m,n)$ vertices of $P$ for any $m$ and $n$. The second subsection deals with the dimension of the linear transformations preserving the bilinear form $F$ associated with $P$ for particular cases of $m$ and $n$.

\subsubsection{Degrees of freedom on the vertices}\label{novela}

The way in which the vertices $A_1, A_{1+2(m,n)}, \ldots , A_{1+2(r-1)(m,n)}$ and their perpendicular spaces $A^\perp_1, A^\perp_{1+2(m,n)}, \ldots , A^\perp_{1+2(r-1)(m,n)}$ are distributed depends on the parity of $m/(m,n)$:

\begin{itemize}
	\item If $m/(m,n)$ is odd, then for each $i$ there is a $j$ such that $A^{\perp}_{1+2i(m,n)}$ is ``centered between'' $A_{1+2j(m,n)}$ and $A_{1+2(j+1)(m,n)}$.
	
	\item If $m/(m,n)$ is even, then for each $i$ there is a $j$ such that $A^\perp_{1+2i(m,n)}$ is ``centered at'' $A_{1+2j(m,n)}$. 
\end{itemize}

See Figure \ref{types-drama}.\\

\textbf{Note:} If $k$ is even, then $m$ must be odd, so the second case ($m/(m,n)$ even) doesn't happen.

\begin{defi}\label{drama-defi}
Let $\boldsymbol{D(m,n,k)}$ denote the number of vertices of the form $A_{1+2j(m,n)}$ that are contained in $A_1^\perp$.
\end{defi}

For instance, Figure \ref{types-drama} shows $D(4, 12, 3)= 0$, while $D(8, 12, 3)= 1$.  \\

\begin{center}
	\begin{minipage}{7cm}
	\includegraphics[scale=0.33]{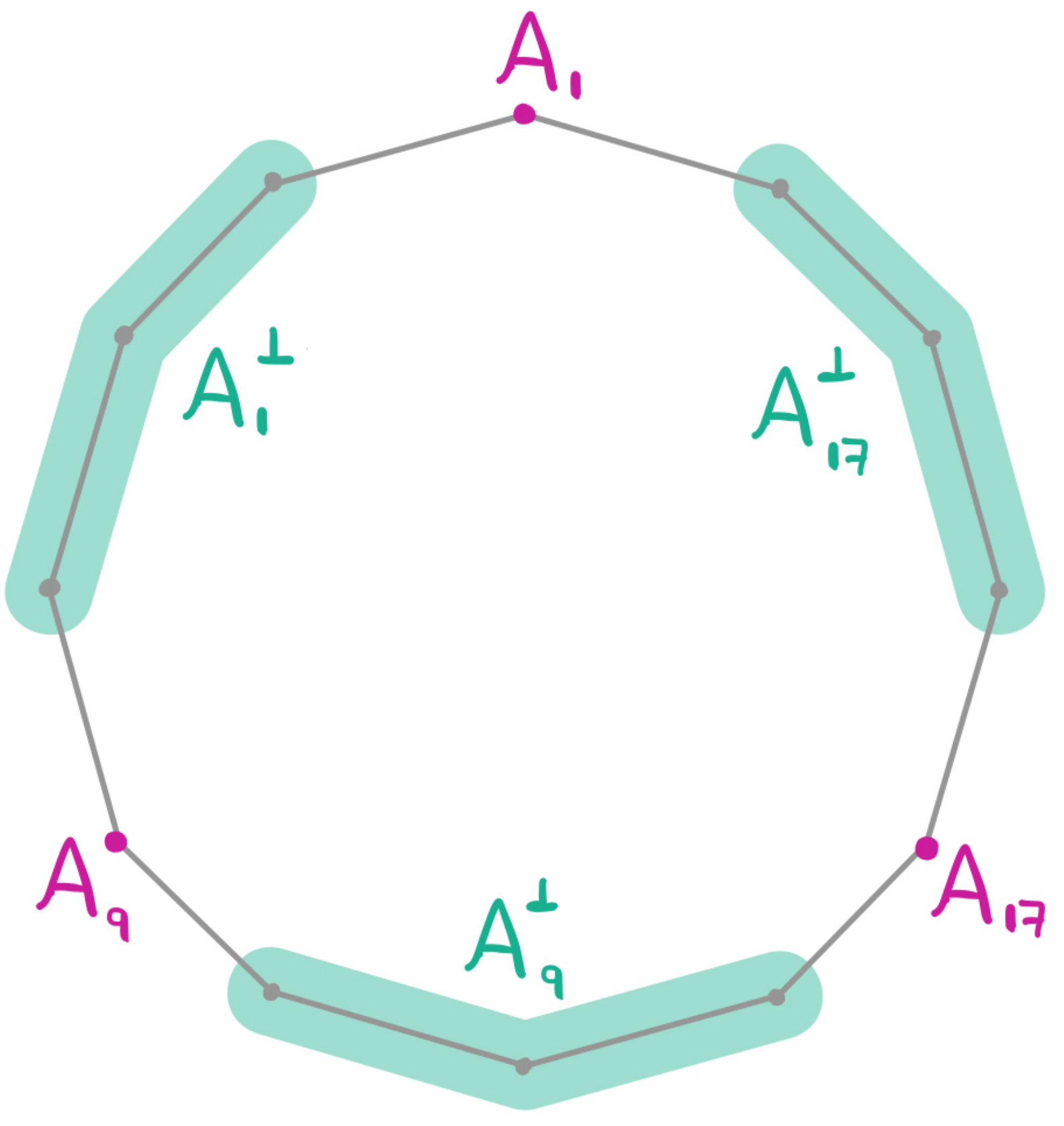}
	\end{minipage}
	\begin{minipage}{7cm}
	\includegraphics[scale=0.33]{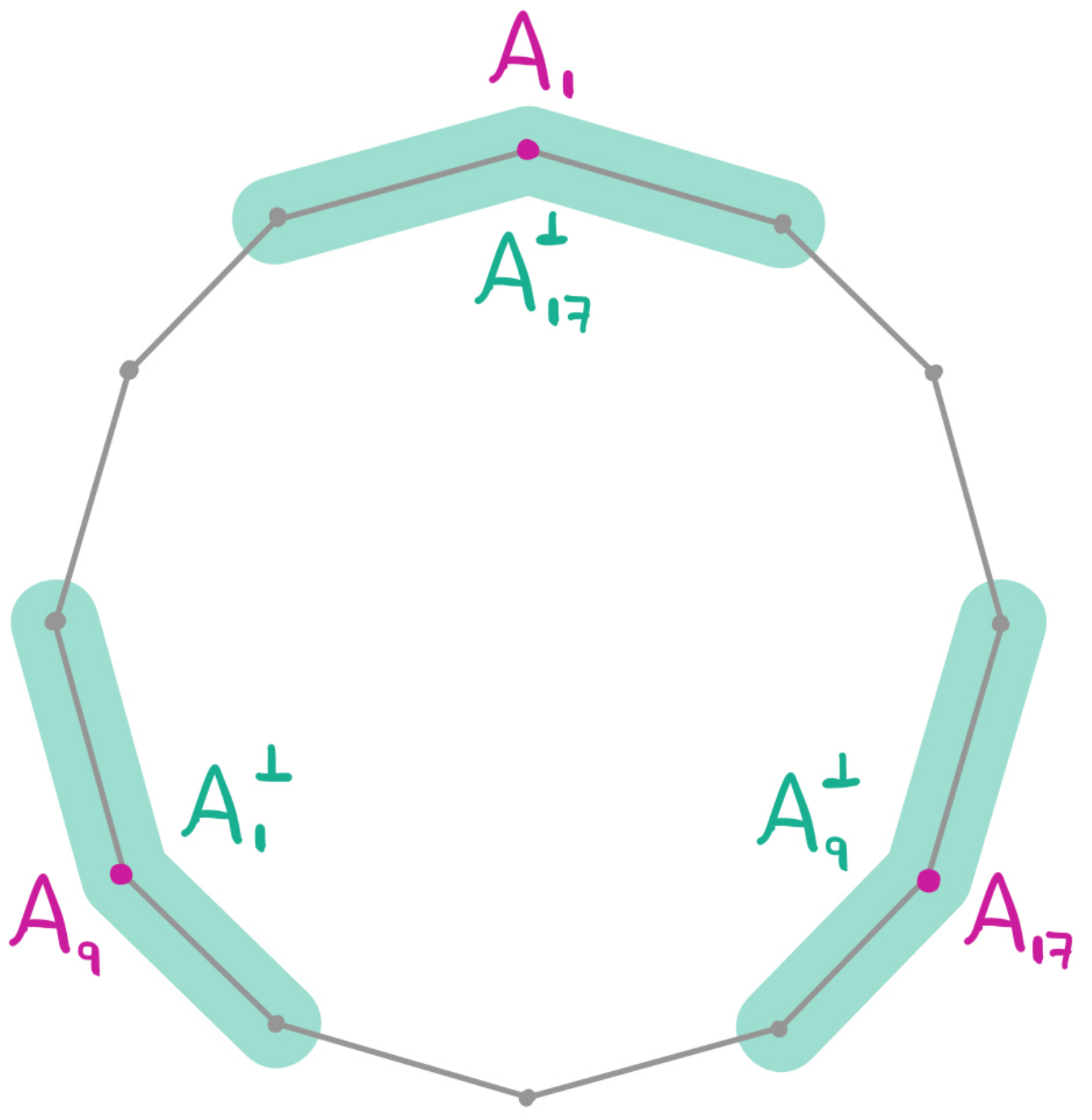}
	\end{minipage}
\figcaption{The picture on the left depicts a $4$-self-dual $12$-gon in $\mathbb{CP}^3$, so $m/(m,n)=1$ is odd. The image on the right shows an $8$-self-dual $12$-gon in $\mathbb{CP}^3$, so $m/(m,n)=2$ is even. The vertices in the shaded areas span the corresponding $A_i^\perp$.}\label{types-drama}
\end{center}

Before deriving a direct formula to compute $D(m,n,k)$, the following proposition shows how this quantity is closely related to the degrees of freedom we have for choosing the vertices $A_1, A_3, \ldots , A_{2(m,n)-1}$.\\

\begin{prop}\label{freedom} Given $P=(A_1, \ldots , A_{2n-1}) \in {\cal M}_{m,n,k}$, the degree of freedom to choose the first $(m,n)$ vertices of $P$ is $$\dfrac{(m,n)(k+D(m,n,k))}{2}.$$
\end{prop}

\begin{proof}
Each vertex $A_i$ can, in principle, be chosen from a Zariski (non-empty) open set of $\mathbb{CP}^k$. But each $A_i$ also determines the perpendicular space $A_i^\perp= \text{span}\{A_{i-(k-1)}, \ldots , A_{i+(k-1)}\}$, imposing $k-D(m,n,k)$ restrictions over the rest of the vertices. Nevertheless, only half of these restrictions count as $A_i\in A_j^\perp  \Leftrightarrow A_{j} \in A_{i}^\perp$ for each $i,j$. \\

Hence, the vertices $A_1, \ldots A_{2(m,n)-1}$ can be chosen from a space of dimension:

$$(m,n)\cdot k - \dfrac{(m,n)(k-D(m,n,k))}{2}= \dfrac{(m,n)(k+D(m,n,k))}{2}.$$
\end{proof}

And here is a formula to compute $D(m,n,k)$:

\begin{lem}\label{compute-drama} The value of $D(m,n,k)$ (Definition \ref{drama-defi}) is given by:
	\begin{itemize}
		\item If $m/(m,n)$ is odd, then $D(m,n,k)= 2\left\lfloor \dfrac{k}{2(m,n)} + \dfrac{1}{2} \right\rfloor $.
		\item If $m/(m,n)$ is even, then $D(m,n,k)= 2\left\lfloor \dfrac{k}{2(m,n)} \right\rfloor +1$.
	\end{itemize}
\end{lem}

\begin{proof} Notice that $D(m,n,k)$ should satisfy $$(D(m,n,k)-1)(m,n)+1 \leq k \leq (D(m,n,k)+1)(m,n) +1,$$ 

with the first inequality coming from the fact that $A_1^\perp$ contains at least $D(m,n,k)$-vertices of the form $A_{1+2\ell(m,n)}$, and the second one indicating $A_1^\perp$ doesn't contain more than $D(m,n,k)+2$ of such vertices (see Figure \ref{drama}). 

	\begin{center}
%	\centering
	\includegraphics[scale=0.45]{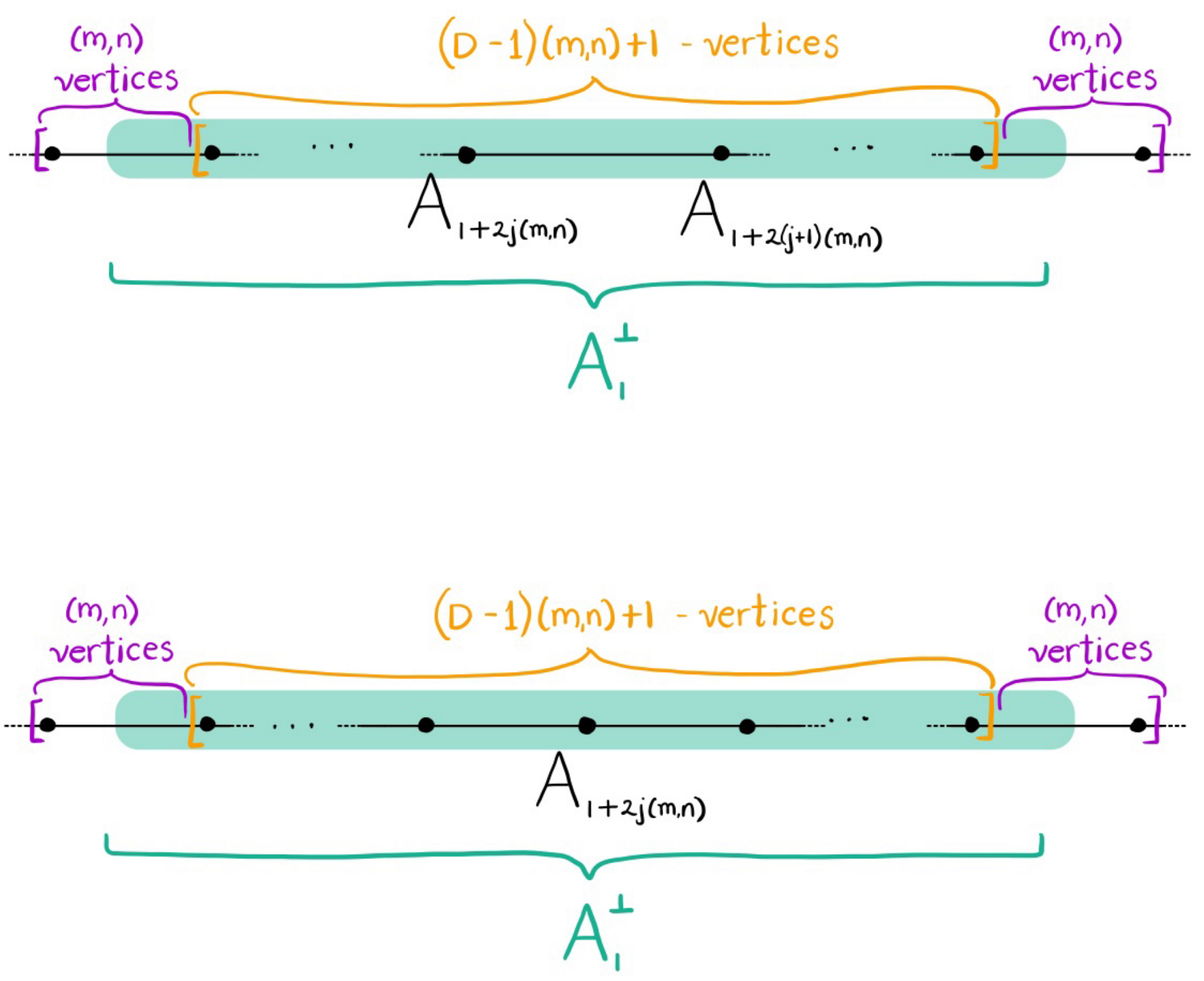}
	\figcaption{The picture above shows the distribution of the vertices of the form $A_{1+2i(m,n)}$ when $m/(m,n)$ is odd, while the figure below shows the case when $m/(m,n)$ is even.} \label{drama}
	\end{center}

Equivalently

\begin{equation} \label{dramamometer}
D(m,n,k)-1 < \dfrac{k}{(m,n)} < D(m,n,k)+1.
\end{equation}

Now, depending on the parity of $m/(m,n)$:
	\begin{itemize}
		\item If $m/(m,n)$ is odd, since $A_1^\perp$ is centered between some $A_{1+2j(m,n)}$ and $A_{1+2(j+1)(m,n)}$, then $D(m,n,k)$ must be an even number. Equation (\ref{dramamometer}) indicates $D(m,n,k)$ must be the closest even number to $k/(m,n)$, and hence:
		$$D(m,n,k) = 2\left\lfloor \dfrac{k}{2(m,n)} + \dfrac{1}{2} \right\rfloor. $$
		
		\item On the other hand, if $m/(m,n)$ is even, then $A_1^\perp$ is centered at some $A_{1+2j(m,n)}$ so $D(m,n,k)$ is odd. From equation (\ref{dramamometer}), we deduce $D(m,n,k)$ is the closest odd number to $k/(m,n)$, thus:
		$$D(m,n,k)= 2\left\lfloor \dfrac{k}{2(m,n)} \right\rfloor +1.$$
	\end{itemize}

\end{proof}

\subsubsection{Degrees of restriction given by $F$}

To compute the dimension of the group of linear transformations that preserve a bilinear form $F$ (when it is not symmetric), we restrict ourselves to two particular cases: $r=n/(m,n)=2$ and $r=3$. The general case is still a work in progress. \\

\textbf{Case 1:} $\boldsymbol{n=2m}$\\

If $n=2m$, then $r=n/(m,n)=2$ (the dodecagon illustrated in Figure \ref{r=2-example} is an example of this case). Therefore $G^2=I$,  and hence there is a basis for which the matrix representation of the bilinear form F is:

\begin{equation}\label{simpleF}
F= \Omega_{2f} \bigoplus I_{k+1-2f},
\end{equation}

where $\Omega_{2f}= \begin{pmatrix}
0 & I \\
-I & 0
\end{pmatrix}$, with $I$ the identity matrix of dimension $f\times f$, and $2f< k+1$.

\begin{lem}\label{restriction-r2} The Lie group of the linear transformations that preserve $F$ written with respect to the basis in which $F$ is expressed in Formula (\ref{simpleF}) is 
$$\left\lbrace g = g_s \oplus g_o \vert \hspace{0.2cm} g_s \in \text{Sp}(2f, \mathbb{C}), g_o\in \text{O}(k+1-2f, \mathbb{C}) \right\rbrace.$$
This group has dimension $4f^2-2kf + {k+1 \choose 2}$. 
\end{lem}

\begin{proof}
Write $F$ as $F=F_+ + F_-$, where $F_+$ is symmetric and $F_-$ is skew-symmetric (notice this decomposition is unique). A linear transformation $g$ preserves $F$ if and only if:
$$g^{\mathsmaller T}Fg= F.$$
Thus:
$$g^{\mathsmaller T}F_+g + g^{\mathsmaller T}F_-g = F = F_+ + F_-.$$
Notice $g^{\mathsmaller T}F_+g$ is symmetric, while $g^{\mathsmaller T}F_-g$ is skew-symmetric. Since the decomposition as sum of symmetric and skew-symmetric parts is unique, one gets:

$$g^{\mathsmaller T}F_+g = F_+  \hspace{0.5cm}\text{ and }\hspace{0.5cm} g^{\mathsmaller T}F_-g = F_-.$$

Since $F_+ = 0 \bigoplus I_{k+1-2f}$ and $F_-= \Omega_{2f} \bigoplus 0$, by looking at the symmetric part we know that $$g= \begin{pmatrix}
A & 0 \\
B & D
\end{pmatrix},$$ where $A$, $0$, $B$ and $D$ are block matrices, with $D$ an orthogonal matrix. Adding the information from the skew-symmetric part, we conclude that $$g= \begin{pmatrix}
A & 0 \\
0 & D
\end{pmatrix},$$ where $A\in \text{Sp}(2f, \mathbb{C}), D \in \text{O}(k+1-2f, \mathbb{C})$.

Therefore the dimension of the group of linear transformations preserving the bilinear form (\ref{simpleF}) is: $$f(2f+1)+ \dfrac{(k+1-2f)(k-2f)}{2} = 4f^2-2kf + {k+1 \choose 2}.$$
\end{proof}

We now proceed with the proof of Theorem \ref{r=2}, which describes the dimension of ${\cal M}_{m, 2m, k}$.

\begin{proof} (Of Theorem \ref{r=2}.) Recall that
$$\dim {\cal M}_{m,2m,k} = \substack{\text{Degrees of freedom} \\ \text{we have to choose} \\A_1, A_3, \ldots A_{2m-1}} - \substack{\text{Dimension of group of} \\ \text{linear transformations}\\ \text{that preserve } F.  }$$

According to Proposition \ref{freedom}:
$$\substack{\text{Degrees of freedom} \\ \text{we have to choose} \\A_1, A_3, \ldots A_{2m-1}}= \dfrac{m(k+D(m,2m,k))}{2}.$$

From Lemma \ref{compute-drama}, since $m/(m,n)=1$ is odd, $D(m,2m,k) = 2 \left\lfloor \frac{k}{2m} + \frac{1}{2} \right\rfloor$. Recall $k+3\leq n =2m$, so $0<k<2m$. This is from where we derive two cases:
$$\begin{cases}
0<k<m \hspace{0.6cm}\Rightarrow \hspace{0.2cm}D(m,2m,k)=0, \\
m<k<2m \hspace{0.26cm}\Rightarrow \hspace{0.2cm} D(m,2m,k)=2 .
\end{cases}$$

Thus
$$\substack{\text{Degrees of freedom} \\ \text{we have to choose} \\A_1, A_3, \ldots A_{2m-1}}= \begin{cases} 
\dfrac{mk}{2}\hspace{1.2cm} \text{ if } k<m,\\
\dfrac{m(k+2)}{2} \hspace{0.23cm} \text{ if } k>m.
\end{cases}$$

By Lemma \ref{restriction-r2},  
$$\substack{\text{Dimension of group of} \\ \text{linear transformations}\\ \text{that preserve } F.  }= 4f^2 -2kf + {k+1 \choose 2}.$$

Let ${\cal M}_f$ denote the space of $m$-self-dual $2m$-gons with bilinear form $F= \Omega_{2f}\bigoplus I_{k+1-2f}$. Therefore

$$\dim {\cal M}_f = \begin{cases}
\dfrac{mk}{2} - \left( 4f^2 -2kf + {k+1 \choose 2} \right) \hspace{1cm}\text{ if } k<m, \\
\dfrac{m(k+2)}{2} - \left(4f^2 -2kf + {k+1 \choose 2} \right) \text{ if } k>m.
\end{cases}$$

To find the dimension of the largest component, we must minimize $4f^2 -2kf + {k+1 \choose 2}$ as a function of $f$. Straightforward calculations show this expression attains its minimum when $f=k/4$. If $f_0$ is the closest integer to $k/4$, then
$$f_0 = \begin{cases} \frac{k}{4} \hspace{0.6cm} \text{ if } k\equiv 0 \pmod 4, \\ 
\frac{k\pm1}{4} \hspace{0.2cm}\text{ if } k\equiv \pm 1 \pmod 4,\\
\frac{k\pm 2}{4} \hspace{0.2cm}\text { if } k\equiv 2 \pmod 4.
\end{cases}$$ 

Thus,

$$4f_0^2-2kf_0+{k+1 \choose 2} = \begin{cases} \frac{k(k+2)}{4} \hspace{0.6cm} \text{ if } k\equiv 0 \pmod 4 ,\\ 
\frac{(k+1)^2}{4} \hspace{0.6cm}\text{ if } k\equiv \pm 1 \pmod 4,\\
\frac{k^2+2k+4}{4} \hspace{0.3cm} \text { if } k\equiv 2 \pmod 4.
\end{cases}$$ 

This provides the dimensions listed in the statement of Theorem \ref{r=2}.
\end{proof}

%-----------------------------------------------------------------------------------------------------------------------
%-----------------------------------------------------------------------------------------------------------------------
%-----------------------------------------------------------------------------------------------------------------------

\textbf{Case 2:} $\boldsymbol{r=3}$\\

The polygon shown in Figure \ref{r=3-example} is an example of this case, when $r= n/ (m,n) =3$. If $r=3$, then $G^3=I$, and so

\begin{equation}
F= \bigoplus_{j=1}^{s_1}H_2(e^{i\theta}) \bigoplus_{j=1}^{s_2} H_2(e^{-i\theta}) \bigoplus I_{k+1-2(s_1+s_2)},
\end{equation}

where $\theta=2\pi/3$. Reordering the basis, one can express $F$ as:

\begin{equation}\label{newF}
F= \begin{pmatrix}
0 & W \\ e^{-i\theta}W^{\mathsmaller T} & 0
\end{pmatrix} \bigoplus I_{k+1-2(s_1+s_2)},
\end{equation} 

where $W=\begin{pmatrix} 0 & I_{s_2} \\ e^{i\theta}I_{s_1} & 0 \end{pmatrix}$ and $I_{s_i}$ is the identity block matrix of size $s_i\times s_i$.

\begin{prop}\label{restriction-r3} The Lie group of the linear transformations that preserve $F$, written with respect to the basis in which $F$ is expressed in Formula (\ref{newF}), is: 
$$\left\lbrace h_1 \oplus h_2 \oplus g_o \vert \hspace{0.2cm} h_1 \in \text{Gl}(s_1+s_2, \mathbb{C}), h_2=W^{-1}h_1^{-\mathsmaller T}W, g_o\in \text{O}(k+1-2(s_1+s_2), \mathbb{C}) \right\rbrace.$$
This group has dimension $3(s_1+s_2)^2-(2k+1)(s_1+s_2) + {k+1 \choose 2}$. 
\end{prop}

\begin{proof}
Decompose $F$ into its unique sum of symmetric and skew-symmetric factors: $F=F_+ +F_-$. If $g$ preserves $F$, then it simultaneously preserves $F_+$ and $F_-$. Straightforward computations show that $g=h \bigoplus g_0$, where $h$ preserves $\begin{pmatrix}
0 & W \\ e^{i\theta}W^T & 0
\end{pmatrix}$, and $g_0 \in O(k+1-2(s_1+s_2), \mathbb{C})$. \\

Restricting now the computations for the linear transformations that preserve $\begin{pmatrix}
0 & W \\ e^{i\theta}W^T & 0
\end{pmatrix}$, one finds that $h = \begin{pmatrix}
h_1 & 0 \\ 0 & h_2 
\end{pmatrix}$, where $h_1$, $h_2$ satisfy $h_1^{\mathsmaller T}W h_2 = W$.

\end{proof}

Theorem \ref{r=3} follows then from Propositions \ref{restriction-r3} and \ref{freedom}.

%-------------------------------------------------------------------------------------------------------------------------------------
%-------------------------------------------------------------------------------------------------------------------------------------
%-------------------------------------------------------------------------------------------------------------------------------------

\section{Gale transform in ${\cal M}_{m,n,k}$}\label{gale-transform}

Gale transform in its most classic version is an involution on sets of points in projective space, which arises from the duality between the Grassmanians $\text{Gr}(k+1,n)$ and $\text{Gr}(n-k-1,n)$. In \cite{projective-gale}, D. Eisenbud and S. Popescu recap the evolution of the Gale transform since Coble studied the transformation in full generality back in the earlier 20th century, up to its modern view in Algebraic Geometry. In \cite{gale}, S. Morier-Genoud, V. Ovsienko, R. Schwartz, and S. Tabachnikov use the connections between linear difference equations, Frieze patterns, and spaces of polygons to induce a Gale transform between polygons.\\

Providing a detailed definition of the Gale transform for polygons is out of this paper's scope. For the reader without previous knowledge about the Gale transform, we recommend \cite{gale}. In the meantime, it will be enough to know that, if ${\cal C}_{n,k}$ denotes the moduli space\footnote[5]{Two polygons $P=(A_1,\ldots, A_{2n-1})$, $P'=(A'_1,\ldots, A'_{2n-1})$ are equivalent if there is a projective transformation $\varphi$ such that $\varphi(A_i)=A'_i$ for all $i$.} of $n$-gons in $\mathbb{P}^k$ with $(n,k)=1$, then the \emph{\textbf{Gale transform}} is an isomorphism between ${\cal C}_{n,k}$ and ${\cal C}_{n, n-k-2}$ that satisfies the following property:

\begin{lem}\label{translation}
Let $P$ be an $n$-gon in $\mathbb{CP}^k$, with $(n,k)=1$, and let $w:=n-k-2 \geq 1$. If ${\cal G}(P)$ is the Gale transform of $P$ (which is an $n$-gon in $\mathbb{CP}^w$), then

$$P  \text{ is } m\text{-self-dual } \Leftrightarrow {\cal G}(P) \text{ is } (m-n)\text{-self-dual} \Leftrightarrow {\cal G}(P) \text{ is } (m+n)\text{-self-dual}.$$

That is, the Gale transform provides a bijective correspondence between ${\cal M}_{m, n, k}$ and ${\cal M}_{(m-n), n, w}= {\cal M}_{(m+n), n, w}$.
\end{lem}

From \cite{selfdual}, we know all pentagons are $5$-self-dual in $\mathbb{RP}^2$. This lemma provides a generalization of this fact for higher dimensions:

\begin{proof} (Of Theorem \ref{k+3})
If $P$ is a $(k+3)$-gons in $\mathbb{CP}^k$, then ${\cal G}(P)$ is a $(k+3)$-gon in $\mathbb{CP}^1$. All polygons are $0$-self-dual in $\mathbb{CP}^1$ (since the duality works as the projective transformation that sends $P$ to $P^*$). By Lemma \ref{translation}, $P$ is then $(k+3)$-self-dual.
\end{proof}

Applying Lemma \ref{translation} to Theorem \ref{m=n}, one gets a proof for Theorem \ref{m=0}. As promised in the introduction, here is an alternative proof of Theorem \ref{m=0} using a constructive approach, just like in the previous sections.

\begin{proof}(of Theorem \ref{m=0})
By equation \ref{perpendicular} from Lemma \ref{rotation}, if $m=0$ then F($A_i$, $A_i$)=0, for every $i$. This means the bilinear form $F$ must be completely symplectic, that is,
$$F= \Omega_{k+1} = \begin{pmatrix}
0 & I \\
-I & 0
\end{pmatrix} \text{ where } I \text{ is the identity matrix of size } \dfrac{k+1}{2}\times \dfrac{k+1}{2}.$$

As usual, 

$$\dim {\cal M}_{0,n,k} = \substack{\text{Degrees of freedom} \\ \text{we have to choose} \\ A_1, \ldots , A_{2n-1}} - \substack{\text{Dimension of group of} \\ \text{linear transformations}\\ \text{that preserve } \Omega_{k+1}.  }$$

By Definition \ref{drama-defi}, $D(0,n,k)=1$. Using the same argument from the proof of Proposition \ref{freedom}, one gets that:

$$\substack{\text{Degrees of freedom we have to choose} \\ A_1, \ldots , A_{2n-1}} = \dfrac{n(k+1)}{2}.$$

And from Lemma \ref{restriction-r2}, 

$$\substack{\text{Dimension of group of linear transformations}\\ \text{that preserve } \Omega_{k+1}  } = \dfrac{(k+1)(k+2)}{2}.$$

Thus,

$$\dim {\cal M}_{0,n,k} = \dfrac{n(k+1)}{2} - \dfrac{(k+1)(k+2)}{2} = \dfrac{(k+1)(n-k-2)}{2}.$$

\end{proof}

Consider, for instance, the Theorem \ref{Fuchs-Tabach} by Fuchs and Tabachnikov in \cite{selfdual} mentioned in the introduction of this paper. The Gale transform allows us to rephrase this theorem for $(k+4)$-gons.

\begin{thm} Let $k$ be odd. If $(m,k+4)=1$, then $ {\cal M}_{m,k+4,k}$ consists of the class of regular $(k+4)$-gons. If $(m,k+4)>1$ and $m< 2(k+4)$ then $\dim {\cal M}_{m, k+4, k}= (m,k+4)-1$. Finally, $\dim {\cal M}_{0,k+4, k}= k+1$.\footnote[7]{In order to guarantee $(k+4,k)=1$, we require $k$ to be odd (and therefore $n=k+4$ is also odd). For that reason, the case $n=2m$ doesn't occur.}.
\end{thm}

Notice that the last line of the previous Theorem is consistent with Theorem \ref{m=0}. To have some context, recall $\dim {\cal C}_{n,k}=k(n-k-2)$. Then, the moduli space of $m$-self-dual $(k+4)$-gons in $\mathbb{CP}^{k}$ is a subset of ${\cal C}_{k+4, k}$, which has dimension $\dim {\cal C}_{k+4, k}= 2k$.\\

\section{Pentagram map in higher dimensions}\label{Pentagram-section}

The pentagram map $T$ was first introduced by Richard Schwartz in 1992 (\cite{Schwartz-pentagram}). Given an $n$-gon $P = (A_1, A_3, \ldots , A_{2n-1})$ in $\mathbb{RP}^2$, the image of $P$ under the pentagram map, $T(P)$ has vertices $(W_1, W_3, \ldots , W_{2n-1})$ where $$W_\ell = (A_\ell A_{\ell+4})\cap(A_{\ell+2} A_{\ell+6}) .$$

	\begin{center}
	\includegraphics[scale=0.7]{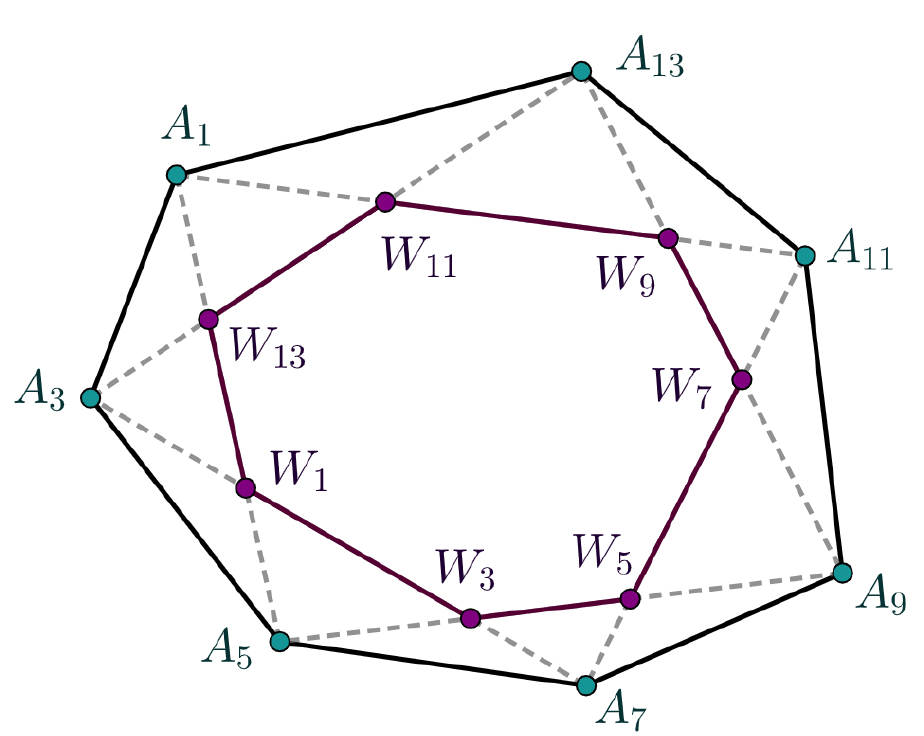}
	\figcaption{A polygon $P=(A_1, \ldots , A_{2n-1})$ and the image of $P$ under the Pentagram map, $T(P)=(W_1, \ldots , W_{2n-1})$.}
	\label{Pentagram}
	\end{center}
\vspace{0.3cm}

This map, defined in the moduli space of polygons in $\mathbb{RP}^2$, has been extensively studied since then, as it exhibits integrability behavior (see \cite{quasiperiodic}, \cite{pentagram-integrable}, \cite{soloviev-integrability}). Several papers have studied the connections between this map to Cluster algebras (\cite{glickpentagram}), Frieze patterns (\cite{friezecluster}), $T$-systems(\cite{kedem2015t}), Octahedral tilings, and Methods of Condensation (\cite{schwartzcondensation}). When we consider the continuous limit of a polygon, the limit of the pentagram map coincides with the classical Boussinesq equation, which is a classic example of an integrable partial differential equation (\cite{pentagram-integrable}). The study of the pentagram map has been extended to twisted, corrugated and dented polygons (\cite{higher-pentagram}, \cite{dented}, \cite{ovsienko-pentagram}, \cite{pentagram-integrals}). \\

Gloria Mar\'i Beffa, Boris Khesin, and Fedor Soloviev have found generalizations of this map to higher dimensions (\cite{beffa-pentagram}, \cite{khesin-pentagram}, \cite{dented}). In $\mathbb{RP}^k$, there are multiple ``Pentagram maps'': given two multi-index $I$, $J$, there is a pentagram map $T_{I,J}$. The $I$ multi-index governs the structure taken to construct the hyperdiagonals: the subindices $i_\ell$ indicate how many vertices to skip (see Figure \ref{hyperdiagonal}). The multi-index $J$ controls which hyperdiagonals to intersect in order to produce a vertex. More concretely:

\begin{defi} Let $I = (i_1, i_2, \ldots, i_{k-1})$ and $J = (j_1, j_2, \ldots, j_{k-1})$ be two $(k-1)$-tuples of natural numbers. If $P=(A_1, A_3, \ldots, A_{2n-1})$ is an $n$-gon in $\mathbb{RP}^k$, then a \emph{$I$-diagonal hyperplane} $D_\ell$ is the hyperplane:

$$D_\ell := \text{span}\{A_\ell, A_{\ell+2i_1}, A_{\ell+2(i_1+i_2)}, \ldots , A_{\ell+2(i_1+i_2+ \ldots +i_{k-1})}\}.$$

Then the image of $P$ under the \emph{\textbf{Generalized Pentagram Map}} $\boldsymbol{T_{I,J}}$ is the polygon with vertices $W_1, W_3, \ldots, W_{2n-1}$ where

$$W_\ell := D_\ell \cap D_{\ell+ 2j_1} \cap D_{\ell+2(j_1 + j_2)} \cap \ldots \cap D_{\ell+2(j_1+j_2+\ldots +j_{k-1})}.$$
\end{defi}

	\begin{figure}
	\centering
	\includegraphics[scale=0.7]{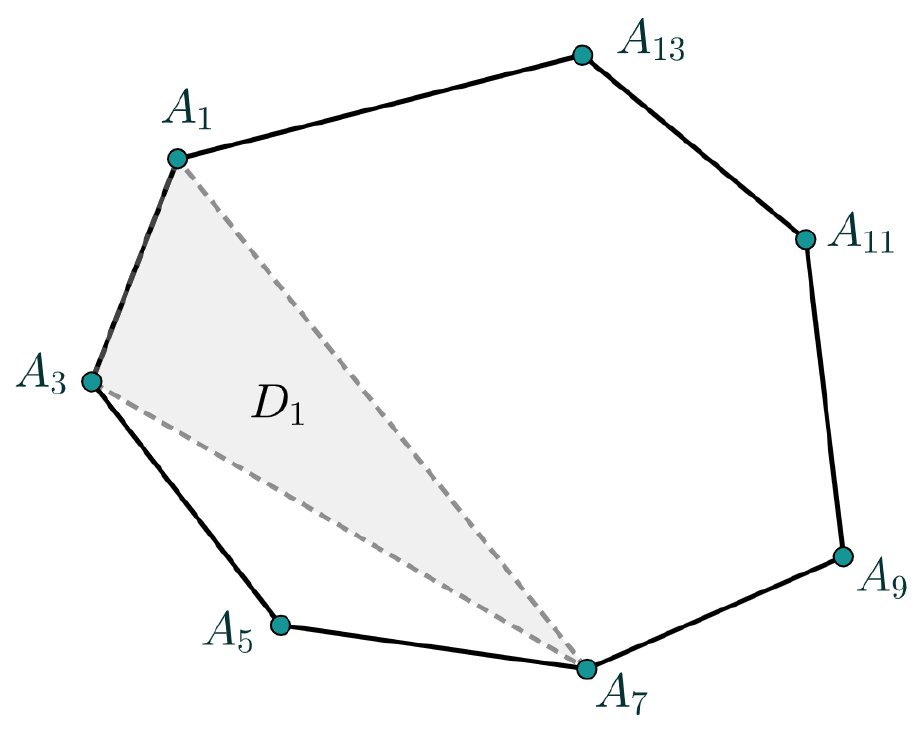}
	\caption{The following picture shows the $I$-diagonal $D_1$, for $I=(1,2)$.}
	\label{hyperdiagonal}
	\end{figure}
For instance, the classic pentagram map corresponds to $T_{(2),(1)}$ (skip one vertex, and take the intersection of two consecutive diagonals) (See Figure \ref{Pentagram}).\\

Clebsch's Theorem (\cite{clebsch}) says that pentagons are invariant under the pentagram map: that is, for any pentagon $P$ in $\mathbb{RP}^2$ there is a projective transformation $\varphi$ such that $\varphi(T(P)) = P$. One can check using computations, that a similar phenomenon occurs in $\mathbb{RP}^3$.

\begin{prop}
Given any hexagon $H$ in $\mathbb{RP}^3$, there is a projective transformation $\varphi$ such that $\varphi(T_{(1,2),(1,1)}(H)) = H$. \\
\end{prop}

Computer-based experiments suggest that for any $k\geq 2 \in \mathbb{N}$, and any $(k+3)$-gon $P$ in $\mathbb{RP}^{k}$, there are multi-indices $I, J$ and projective transformation $\varphi$ such that $\varphi(T_{I,J}(P)) = P$

\begin{conj}
Let $P$ be a $(k+3)$-gon in $\mathbb{RP}^k$, and $J=(1, \ldots, 1)$.
\begin{enumerate}
	\item If $k$ is even number and $I=(1, \ldots, 1, 2, 1, \ldots , 1)$ (same number of $1$'s at the beginning and at the end of the multi-index), or
	\item If $k$ is odd and $I = (1, 1, \ldots, 1, 2, 1, \ldots, 1)$ (one more $1$'s at the beginning of the multi-index than at the end)
\end{enumerate}
then $P$ and $T_{I,J}(P)$ are projectively equivalent.
\end{conj}

While there is no proof of this conjecture, there are computational results that support this for up to dimension $k=200$.

\section*{Acknowledgments} 
The author would like to thank Sergei Tabachnikov for all the fruitful discussion and valuable feedback. Many thanks to Jack Huizenga for sharing the paper by R.A. Horn and V.V. Sergeichuk \cite{canonical}, which was crucial for understanding the bilinear forms associated with dual polygons. Thanks to Gil Bor for the stimulating discussions.  This research is supported by the U.S. National Science Foundation through the grant DMS-2005444.

\bibliographystyle{abbrv}
\bibliography{mselfdual_polygons}

\end{document}